\documentclass[12pt, letterpaper]{article}

\usepackage{amsmath,amsfonts}
\usepackage{algorithmic}
\usepackage{algorithm}
\usepackage{array}
\usepackage[caption=false,font=normalsize,labelfont=sf,textfont=sf]{subfig}
\usepackage{textcomp}
\usepackage{stfloats}
\usepackage{url}
\usepackage{verbatim}
\usepackage{graphicx}
\usepackage{cite}

\usepackage{caption} 

\usepackage{pgfplots}

\usepackage{amsthm}

\advance \topmargin by -0.27in

\newcommand{\C}{ {\mathcal {C}} }
\newcommand{\Z}{\mathbb{Z}}
\newcommand{\F}{\mathbb{F}}

\newcommand{\vspan}[1]{\left\langle {#1} \right\rangle}

\renewcommand{\Bbb}{\mathbb}

\theoremstyle{plain} \newtheorem{theorem}{Theorem}
\theoremstyle{plain} 
\theoremstyle{plain} \newtheorem{lemma}[theorem]{Lemma}
\theoremstyle{plain} 
\theoremstyle{plain} \newtheorem{rem}[]{Remark}
\theoremstyle{plain} \newtheorem{corollary}[theorem]{Corollary}
\theoremstyle{plain} \newtheorem{example}[theorem]{Example}

\author{Tim L. Alderson\\
	Department of Mathematics and Statistics\\
	University of New Brunswick Saint John\\
	Saint John, NB, E2L 4L5\\
	Canada}
\title{Bounds on MLDR Codes Over $\Z_{p^t}$}
\date{\today}

\begin{document}
\maketitle

\begin{abstract}
	Upper bounds on the minimum Lee distance of codes that are linear over $\Z_q$, $q=p^t$, $p$ prime are discussed. The bounds are Singleton like, depending on the length, rank, and alphabet size of the code. Codes meeting such bounds are referred to as Maximum Lee Distance with respect to Rank (MLDR) Codes. We present some new bounds on MLDR codes, using combinatorial arguments. In the context of MLDR codes, our work provides improvements over existing bounds in the literature.  
	

\end{abstract}

\section{Introduction}

The Lee metric, introduced by Lee in 1958 as an alternative to the Hamming metric \cite{Lee1958}, was initially applied to various noisy communication channels, particularly those employing phase-shift keying modulation \cite{Nakamura1979}. Over time, codes designed for the Lee metric have found applications in constrained and partial-response channels \cite{Byrne2002}, interleaving schemes \cite{Blaum1998}, orthogonal frequency-division multiplexing (OFDM) \cite{Schmidt2007}, multidimensional burst-error correction \cite{Etzion2009a}, and error correction for flash memories \cite{Barg2010}. The study of optimal Lee codes has gained increasing attention, driven by these diverse applications.

As with the Hamming metric, a fundamental problem is determining upper bounds on $d_L(\C)$, the minimum Lee distance of a code $\C$ with fixed size, length, and alphabet. One of the earliest such bounds was established by Wyner and Graham \cite{Wyner1968}. Later, Chiang and Wolf \cite{MR0342262} provided a bound on free codes that, in many cases, improved upon that of Wyner and Graham. This bound was subsequently shown to hold for (not necessarily linear) codes of integral $p^t$-type in \cite{Alderson2013}. More recently, Byrne and Weger \cite{Byrne2023} extended the Chiang–Wolf bound to the setting of (not necessarily free) linear codes over the integer residue ring $\Z_{p^t}$, where $p$ is prime.

In this work, we present new bounds on the minimum Lee distance of linear codes of fixed rank and length over $\Z_{p^t}$. These bounds not only generalize existing results but also provide improvements over each of the aforementioned bounds. 

\section{Preliminaries}\label{sec: Prelim}

\subsection{MDS Codes and Singleton Defect} \label{sec: MDS Codes and Singleton}
For $n\geq k$, an $(n,k,d)_q$ code $\C$ is a collection of $q^k$
$n$-tuples (or \textit{codewords}) over an alphabet $A$ of size $q$
such that the minimum (Hamming) distance between any two codewords
of $\C$ is $d$.  In
the special case that $A=\F_q$ (the finite field of order $q$) and
$\C$ is a vector space of dimension $k$, $\C$ is a (classical) \textit{linear}
$ [n,k,d ]_q$-code. Without reference to minimum distance, a $q$-nary linear code of length $n$ and dimension $k$ is denoted an $[n,k]_q$-code.\\

The Singleton bound \cite{Singleton1964} dictates that an $(n,k,d)_q$ code
satisfies $d\leq n-k+1$. Codes meeting the Singleton bound are called maximum distance separable (MDS) codes.  The maximum length of a linear $[n,k]_q$-MDS code is denoted $m(k,q)$. 
In 1952,  Bush \cite{Bush1952} established that if $k\ge q$ then $m(k,q) = k+1$. For $k<q$ there is a long standing conjecture regarding linear MDS codes: Every linear $[n,k,n-k+1]_q$ MDS code  with $1 < k < q$ satisfies $n \le q + 1$, except when $q$ is even and $k = 3$ or $k = q - 1$ in which case $n \le q + 2$. This conjecture is called the Main Conjecture on Linear MDS Codes, and has been shown to hold in many cases,  (see e.g. \cite{Ball2012a}).

The \textit{Singleton defect} of an
$(n,k,d)_q$ code is \[\text{def}(\C)=n-k+1-d.\] An $(n,k,d)_q$ code $\C$ with
$\text{def}(\C)=0$ is therefore an  MDS code.  Codes of Singleton defect $1$ are called Almost-MDS
(AMDS) codes \cite{MR1409442}, and more generally, those of Singleton defect $s$ are denoted A$^s$MDS codes.

A fundamental problem in coding theory is that of determining the
maximum length of a code with $k$, $q$, and $\text{def}(\C)$ fixed. 

The study of MDS codes, is a major research area. In the classical setting, several properties have been established. In particular, the recent work of Ball \cite{Ball2012a} established that for $p$ prime, linear MDS codes of dimension $k<p$ over $\F_p$ have length at most $p+1$, and attain this bound if and only if the code corresponds to a normal rational curve in $PG(k-1,q)$. Consequently, the bound is only attained by codes with a generator matrix equivalent to the matrix with columns
\[S=\{(1,t,t^2,\ldots,t^{k-1}) \mid t\in \F_p\}\cup \{(0,0,\ldots,0,1)\}.\] 

More generally, if $\C$ is a linear $[n,2]_q$-code with generator matrix $G$, then $\C$ has a codeword of weight $n-\alpha$ if and only if there exists a subset, $S$ of $\alpha$ columns of $G$ that are pairwise linearly dependent, and $S$ is maximal (in that the union of $S$ with any further column of $G$ no longer has this property). It follows that if def$(\C)=s$ then $n\le (q+1)(s+1)$. Inductively, for $k\ge 2$, a linear $[n,k,d]_q$ code $\C$ with $\text{def}(\C)=\delta$ satisfies
\begin{equation}\label{eqn: linear code length bounded by defect}
	k+\delta \le	n\le (\delta+1)(q+1)+k-2.
\end{equation}

In what follows we investigate optimality conditions analogous to the Singleton bound, in the setting of the Lee metric for linear codes over the integer residue ring $\Z_q$, where $q=p^t$ for some prime $p$.

\subsection{Linear Codes Over $\Z_q$, $q=p^t$}
Linear codes  over $\Z_q$, $q=p^t$ are sub-modules over the integer residue ring $\Z_{q}$. 
A $\Z_{q}$ module $\C$ of $\left(\Z_{q}\right)^n$ is called a linear code of length $n$.
An analogue of the (classical) dimension is the $p^t$-type of the code, defined by:
\begin{equation}\label{eqn: kappa}
	\kappa=\log _{p^t}(|\C|) . 
\end{equation}

From the basic theory, any $\Z_{p^t}$ module $\C$ satisfies an isomorphism of the form 

\begin{equation} \label{eqn: C isomsm}
	\C\cong \left(\Z_{p^t}\right)^{k_1} \times\left(\Z_{p^{t-1}}\right)^{k_2} \times \cdots \times(\Z_{p})^{k_t},
\end{equation}

where the parameter $k_1$ is called the free rank of the code and $K=\sum_{i=1}^s k_i$ is called its rank. If $K=k_1$ then the code is said to be free. 
As in the classical case, if $\C \subseteq\left(\Z_{p^t}\right)^n$ is a linear code then we call a matrix $G$ a generator of $\C$ if its row-span is $\C$. 
It is helpful to consider these matrices in their systematic form.

If $\C \subseteq \left(\Z_{p^t}\right)^n$ is a linear code as in (\ref{eqn: C isomsm}), with rank $K$, then $\C$ is equivalent (up to a permutation of coordinate positions) to a code having the following systematic generator matrix $G \in\left(\Z_{p^t}\right)^{K \times n}$
$$
G=\left(\begin{array}{cccccc}
	I_{k_1} & A_{1,2} & A_{1,3} & \cdots & A_{1, t} & A_{1, t+1} \\
	0 & p I_{k_2} & p A_{2,3} & \cdots & p A_{2, t} & p A_{2, s+1} \\
	0 & 0 & p^2 I_{k_3} & \cdots & p^2 A_{3, t} & p^2 A_{3, t+1} \\
	\vdots & \vdots & \vdots & & \vdots & \vdots \\
	0 & 0 & 0 & \cdots & p^{t-1} I_{k_t} & p^{t-1} A_{t, t+1}
\end{array}\right),
$$
where  $A_{i, t+1} \in\left(\Z_{p^{t+1-i}}\right)^{k_i \times(n-K)}, A_{i, j} \in\left(\Z_{p^{t+1-i}}\right)^{k_i \times k_j}$ for $j \leq t$, and  $I_{k_\alpha}$ denotes the $k_\alpha \times k_\alpha$ identity matrix.

The socle of a linear code is the sum of its minimal sub-modules; hence, in the current context, the socle of $\C$ is $ S(\C) =\left\langle p^{t-1}\right\rangle \cap \C$.  In an unpublished manuscript \cite{Horimoto-manuscript200}, Horimoto and Shiromoto establish that a linear code has the same length, rank, and minimum distance as its socle; the same result appears in more contemporary works (see e.g. \cite{Kalachi2021}).

In the setting of linear codes over $\Z_{p^t}$ we have the following analogue of the Singleton bound, (see e.g. \cite{Dougherty2017}).

\begin{lemma}[Singleton Bound for Linear Codes over $\Z_{p^t}$] \label{lem: Singleton MDR}
	Let $p$ be prime. 	If $\C \subseteq\left(\Z_{p^t}\right)^n$ is a linear code of rank $K$ and minimum Hamming distance $d_H(C)=d$, then
	$$
	d \leq n-K+1.
	$$
\end{lemma}


\subsection{Lee Weight}\label{subsec: Lee Weight}
The \emph{Lee weight} of an element $a\in\Z_q$ is given by $w_L (a)=\min\{a, q-a\}$. Given an element $c=(c_1,c_2,\ldots,c_n)\in \Z_q^n$, the \emph{Lee weight} of $c$, denoted $w_L(c)$  is given by 
\[
w_L (c)=\sum_{i=1}^nw_L (c_i).
\]
For $c,c'\in \Z_q^n$, the \emph{Lee distance} $d_L(c,c')$ between $c$ and $c'$ is defined to be the Lee weight of their difference, 
\[
d_L(c, c')  = w_L(c-c').
\]

The maximum Lee weight of any element in $\Z_q$ shall be denoted by $M_L(q)$. In particular, we have  
\begin{equation}
	M_L(q)=\left\lfloor\frac{q}{2}\right\rfloor.
\end{equation}  We shall also denote the average nonzero Lee weight of $\Z_q$ by $\mu_q$.  Simple counting shows 

\begin{equation}
	\mu_q= \left\{\begin{array}{ll} \frac{q^2}{4(q-1)} & \text{if $q$ even}\\ \frac{q+1}{4} & \text{if $q$ odd.}\end{array}\right. .
\end{equation} 

Given a linear $ [n,k,d]_q$ code over $\Z_q$, $d_L(\C)$ shall denote the \textit{minimum Lee distance} of $\C$, that is 
\begin{align*}
	d_L(\C) & =\min\{d_L(c, c')\mid c, c'\in \C, c\neq c'\}\\[3pt]
	& =\min\{w_L(c)\mid c \in \C, c\neq 0\}
\end{align*}

Given a specific code $\C$, one can determine $d_L(\C)$ directly, for instance, through brute-force enumeration. Alternatively, methods such as filtration \cite{Bariffi2023} can be used to analyze the entries of a given generator matrix and derive upper bounds on $d_L(\C)$. In this work, however, we focus on bounds that depend only on the length $n$, rank $K$, and alphabet size $q = p^t$.

To frame our discussion, we introduce the following key parameter:

\begin{equation}\label{phi for MLD}
	\Phi(n,K,p^t) = \max\{d_{L}(\C) \mid \C  \subseteq\left(\Z_{p^t} \right)^n \text{ is a linear code of rank }   K\ge 1\}
\end{equation} 

When $t=1$, the rank of the code coincides with its (classical) dimension. To ensure clarity, in the sequel we will use $k$ to denote dimension and $K$ to denote rank. A code $\C \subseteq\left(\Z_{p^t} \right)^n$ of rank $K$ satisfying $d_L(\C) =\Phi(n,K,p^t)$ is called an MLDR (Maximum Lee Distance with respect to Rank) code \cite{Dougherty}. We note that finding and classifying MLRD codes (not necessarily over $\Z_{p^t}$) is a challenging problem, even if one restricts to the case of rings of order $4$ \cite{MR3497905}.

\subsection{Some Established Bounds on $\Phi(n,K,p^t)$, $p$ Prime}\label{subsec: Established bounds}

One of the earliest bounds on the minimum Lee distance of codes was established in response to a conjecture of Lee \cite{Lee1958}, proved in 1968 by Wyner and Graham \cite{Wyner1968}. Their result, when stated in the context of codes over $\Z_{p^t}$, is as follows:

\begin{theorem}[\cite{Wyner1968}] \label{thm: Wyner and Graham}
	Let $p$ be prime. 	If $\C \subseteq\left(\Z_{q}\right)^n$, $q=p^t$ then 
	\[
	d_L(\C) \le n\cdot \mu_q \cdot \frac{q-1}{q}\cdot \frac{|\C|}{|\C|-1}. 
	\]	 	
\end{theorem}

\begin{rem}\label{rem: kappa at least $K/t$}
	From equations (\ref{eqn: kappa}) and (\ref{eqn: C isomsm}), if a code $\mathcal{C} \subseteq\left(\Z_{p^t}\right)^n$ has rank $K$, then its size satisfies $|\C|\ge p^K$ (with the possibility of equality). Consequently, if $\kappa$ denotes the $p^t$-type of $\C$, then we have the relation:
	\[
	K \ge \kappa \ge \frac{K}{t}.
	\]
\end{rem} 

By applying Remark \ref{rem: kappa at least $K/t$} to Theorem \ref{thm: Wyner and Graham}, we obtain the following bound on MLDR codes:

\begin{corollary} \label{cor: Wyner and Graham MLDR}
	Let $p$ be prime. 	If $q=p^t$, then 
	\[
	\Phi(n,K,p^t) \le n\cdot \mu_q \cdot \frac{q-1}{q}\cdot \frac{p^K}{p^K-1}.
	\]	 
\end{corollary}

Shiromoto and Yoshida \cite{Shiromoto2001} provided the following bound, which also follows readily from the Singleton Bound for linear codes over $\Z_{p^t}$.

\begin{lemma}[\cite{Shiromoto2001}]
	Let $p$ be prime. 	If $\mathcal{C} \subseteq\left(\Z_{p^t}\right)^n$ is a code of rank $K$ then 
	\[
	d_L(\C) \le M_L(q) \cdot (n-K+1)
	\]	 
\end{lemma}


Refining a bound on linear codes by Shiromoto \cite{Shiromoto2000}, Alderson and Huntemann \cite{Alderson2013} provide the following bound.

\begin{theorem}[\cite{Alderson2013}]\label{maintheoremgeneralAH}
	Let $p$ be prime. 	If $\mathcal{C} \subseteq\left(\Z_{p^t}\right)^n$, $p^t>3$, is a code of $p^t$-type $\kappa<n$ then
	\[d_L(\C)\leq M_L(q) \cdot \left(n-\lfloor \kappa \rfloor \right).\]
\end{theorem}

With reference to the Remark \ref{rem: kappa at least $K/t$}, Theorem \ref{maintheoremgeneralAH} provides the following bound pertaining to MLDR codes.

\begin{corollary}\label{cor: maintheoremgeneralAH}
	Let $p$ be prime. 	If $p^t>3$ and $n>K$, then
	\[
	\Phi(n,K,p^t) \le M_L(q)\cdot \left(n-\left\lfloor\frac{K}{t}\right\rfloor\right).
	\]	
\end{corollary}

The following bound, established by Chiang and Wolf \cite{MR0342262}, is particularly effective when applied to free codes. 

\begin{theorem}[\cite{MR0342262}]\label{thm: ChiangWolf}
	Let $p$ be prime. 	If $\mathcal{C} \subseteq\left(\Z_{p^t}\right)^n$ is a code with free rank $ k_1$, then
	\[
	d_L(\C) \leq \mu_{p^t} (n-k_1+1).
	\]
\end{theorem}

Since free rank and rank coincide when working over $\Z_p$, Theorem \ref{thm: ChiangWolf} provides the following bound on MLDR codes. 

\begin{corollary}\label{cor: Chiang and Wolf}
	If $p$ is prime, then 	
	\[
	\Phi(n,K,p) \le \mu_p(n-K+1).
	\]
\end{corollary}

Alderson and Huntemann \cite{Alderson2013} later provided the following refinement of Theorem \ref{thm: ChiangWolf}, extending it to the setting of integral $p^t$-type codes.

\begin{theorem}[\cite{Alderson2013}]\label{ChiangWolfgeneral}
	Let $p$ be prime. 	If $\mathcal{C} \subseteq\left(\Z_{p^t}\right)^n$ is a code of integral $p^t$-type $\kappa$, then
	\[
	d_L(\C) \leq \mu_{p^t} (n-\kappa+1).
	\]
\end{theorem}

Using Remark \ref{rem: kappa at least $K/t$}, we obtain the following corollary to Theorem \ref{ChiangWolfgeneral}, which generalizes Corollary \ref{cor: Chiang and Wolf}.

\begin{corollary}\label{cor: ChiangWolfgeneral} 
	Let $p$ be prime. 	If $q=p^t$, then	
	\[
	\Phi(n,K,q) \le \mu_q \left(n-\left\lfloor\frac{K}{t}\right\rfloor+1\right).
	\]
\end{corollary}

More recently, additional bounds on $\Phi(n,K,p^t)$ have been developed. For instance, Bariffi and Weger \cite{Bariffi2023} established the bound  
\[
\Phi(n,K,p^t) \leq p^{t-1} \left\lfloor \frac{p}{2} \right\rfloor (n-K+1).
\]
While this result improves upon some previously known bounds, it is itself subsumed by the following bound established by Byrne and Weger \cite{Byrne2023}.

\begin{theorem}[\cite{Byrne2023}]\label{thm: ByrneWeger} 
	Let $p$ be prime. 	For $q=p^t$,
	\[
	\Phi(n,K,q) \leq p^{t-1} \mu_p (n-K+1).
	\]
\end{theorem}

It is worth noting that Theorem \ref{thm: ByrneWeger} corrects a minor discrepancy in the original result, which appears as Corollary 32 in \cite{Byrne2023}. 

In the following sections, we present further refinements that, in many cases, improve upon the existing bounds on MLDR codes.

\section{New Bounds}

Before developing new bounds for codes over $\Z_q$, where $q=p^t$, we first restrict our discussion to the case $q=p$. As noted in Section \ref{subsec: Lee Weight}, we shall use the dimension parameter $k$ to represent the rank when $t=1$, since in this case, rank and dimension coincide.

\subsection{Bounds for Linear Codes Over $\Z_p$, $p$ Prime }

\begin{lemma} \label{lem:LDimpliesMDS}
	Let $p$ be prime. 	If $\C$ is a linear $[n,k,d]_p$ code with $d_L(\C) >  \left\lfloor \mu_p(n-k) \right\rfloor$, then 
	$\C$ is an MDS code, so $d = n-k+1$. 
\end{lemma}

\begin{proof}
	If $p=2,3$, the result holds trivially, so assume $p\ge 5$. Suppose $x \in \C$ satisfies $w_H(x) \le n-k$, and consider the $(p-1) \times n$ array whose rows consist of the nonzero elements of $\langle x\rangle$. In each column corresponding to a support position of $x$, each nonzero element of $\Z_p$ appears exactly once. Thus, the total Lee weight of the array is bounded above by
	\[
	(n-k) \cdot \frac{p^2-1}{4}.
	\] 
	The result follows by averaging over the $p-1$ rows.
\end{proof}

\begin{lemma} \label{lem:LDimpliesS(C)}
	Let $p$ be prime. 	Let $\C$ be a linear $ [n,k,d]_p$ code and $s\in \Z^+$. If $d_L(\C)>  \left\lfloor\mu_p(n-k-s)\right\rfloor$, then $d\ge n-k-s+1$.\\
	Equivalently, if $d_L(\C)>  \left\lfloor \mu_p(n-k-s) \right\rfloor$ then  $\operatorname{def}(\C)\le s$.
\end{lemma}
\begin{proof}
	If $p=2,3 $ then the result holds trivially, so assume $p\ge5$. By way of contradiction suppose $\C$ has $d_L(\C)> \mu_p(n-k-s)$ and   $x\in \C$ with $w_H(x)\le n-k-s$.  Consider the array whose rows are the nonzero members of $\vspan{x}$.  Counting and averaging the nonzero column weights over the rows, as in the proof of Lemma \ref{lem:LDimpliesMDS} gives the contradiction $d_L(\C)\le \mu_p(n-k-s)$.       
\end{proof}

As an immediate consequence, we have the following.

\begin{corollary} \label{cor: bound for AsMDS prime} Let $p$ be prime.  If $\C$ is a linear $ [n,k,d]_p$ A$^s$MDS code, then
	\begin{equation} \label{eqn: Lee distance and Singleton defect}
		d_L(\C) \le \left\lfloor \mu_p(n-k+1-s) \right\rfloor = \left\lfloor d\cdot \mu_p  \right\rfloor.
	\end{equation}  
\end{corollary}

The following example demonstrates that equality can be attained in (\ref{eqn: Lee distance and Singleton defect}).

\begin{example} \label{example: n=k n=k+1, prime} 
	For any prime $p$ and $k\geq 1$, the code $\C = (\Z_p)^k$ is an $[k,k,1]_p$ MDS code with $d_L(\C) = 1$. If $p \in \{2,3\}$ (when the Hamming and Lee metrics coincide), then $\mu_p = 1$, and equality holds in (\ref{eqn: Lee distance and Singleton defect}).\\

	Similarly, consider the MDS code $\C_1$ over $\Z_p$ with generator matrix 
	\[
	G_1 = [I_k \mid \bar{1}] := \begin{bmatrix}
		1 & 0   & \cdots & 0 & 1 \\
		0 & 1   & \cdots & 0 & 1 \\
		\vdots & \vdots   & \ddots & \vdots & \vdots \\
		0 & 0   & \cdots & 1 & 1
	\end{bmatrix}.
	\]
	In this case, $d_L(\C_1) = 2$, and if $p \in \{2,3\}$, then equality holds in (\ref{eqn: Lee distance and Singleton defect}). Consequently, we obtain \[\Phi(k+1,k,2) = \Phi(k+1,k,3) = 2.\]
\end{example}

From the discussion in Section \ref{sec: MDS Codes and Singleton}, an $[n,k,d]_p$-MDS code with $k\leq p$ satisfies $n\leq p+1$, and if $k>p$, then $n\leq k+1$ (\cite{Ball2012a}, \cite{Bush1952}). Consequently, Lemma \ref{lem:LDimpliesS(C)} leads to the following result.

\begin{corollary} \label{cor: bound for MDS prime} 
	Let $p$ be prime. 	If $n>p+1$ and $k\leq p$, or if $n>k+1$ and $k\geq p$, then  
	\[
	\Phi(n,k,p) \leq \left\lfloor \mu_p(n-k) \right\rfloor.
	\] 
\end{corollary}

The bound in Corollary \ref{cor: bound for MDS prime} can be improved if the codes under consideration are relatively long, in that $n>2p+k$.  

\begin{corollary} \label{cor: floor bound prime case}
	Let $p$ be prime. 	If $k\geq 2$, then
	\[
	\Phi(n,k,p) \leq \mu_p\left(n-k+1- \left\lfloor \frac{n-k+1}{p+1}  \right\rfloor\right).
	\]	
\end{corollary}

\begin{proof}
	Let $\C$ be a linear $[n,k,d]_p$ code with $k\geq 2$. Define $a = \left\lfloor \frac{n-k+1}{p+1}  \right\rfloor$. Then we have $n-k+1\geq a(p+1)$. 
	
	If $\operatorname{def}(\C)\leq a-1$, then equation (\ref{eqn: linear code length bounded by defect}) leads to the contradiction $n\leq a(p+1)+k-2$. Hence, we must have $\operatorname{def}(\C)\geq a$, and the result follows from Corollary \ref{cor: bound for AsMDS prime}.
\end{proof}

%
%

%
%

In the case where $\operatorname{def}(\C)=0$, the validity of the Main Conjecture on MDS Codes over prime fields \cite{Ball2012a} provides the bound $n\leq \max\{k+1,p+1\}$. Consequently, Corollary \ref{cor: floor bound prime case} offers no improvement over Theorem \ref{thm: ByrneWeger} (or Corollary \ref{cor: Chiang and Wolf}) when $\C$ is MDS. However, we present the following refinement in the MDS case.

\begin{theorem} \label{thm: main theorem bound}
	Let $p$ be prime. 	If $\C$ is a linear $ [n,k,d]_p$ code with $k\geq 2$, then either:
	\begin{enumerate}
		\item $\C$ is MDS, in which case 
		\begin{equation}\label{eqn: Ald prime MDS}
			d_L(\C)\leq \left\lfloor \mu_p(n-k+1)\cdot \frac{p-1}{p} \right\rfloor +1,
		\end{equation}
		\item or $\operatorname{def}(\C)=s>0$, in which case 
		\begin{equation} \label{eqn: Ald prime not MDS}
			d_L(\C)\leq \mu_p(n-k+1-s)\leq \mu_p(n-k).
		\end{equation}   
	\end{enumerate}
\end{theorem}

\begin{proof}
	If $\C$ is not MDS, then the result follows from Corollary \ref{cor: bound for AsMDS prime}.  If $\C$ is MDS, then consider the $p\times n$ array $S$, the rows of which comprise the codewords with the first $k-1$ entries equal to $00\cdots 01$. It follows that no two rows of $S$ have a common entry in the final $d=n-k+1$ columns. If $p$ is odd, then summing the Lee weights column-wise over the final $d$ coordinates gives $d\cdot \frac{p^2-1}{4}$, whence averaging shows some row, $x$, satisfies 
	\begin{equation}
		w_L(x) \le \left\lfloor\frac{p^2-1}{4}\cdot d\cdot \frac1p \right\rfloor +1 = \left\lfloor \mu_p(n-k+1) \cdot \frac{p-1}{p} \right\rfloor +1. 
	\end{equation}    	
	If $p=2$ then \textit{mutatis mutandis},  $S$ provides a row, $x$, satisfying
	\begin{equation}
		w_L(x) \le \left\lfloor  \frac{d}{2} \right\rfloor +1 \left(= \left\lfloor \mu_p(n-k+1) \cdot \frac{p-1}{p} \right\rfloor +1\right). 
	\end{equation}    
\end{proof}

\begin{rem}\label{rem: compare with Chiang and Wolf}
	In the present context of codes over $\Z_p$, the bounds of Chiang and Wolf (Corollary \ref{cor: Chiang and Wolf}), Alderson and Huntemann (Corollary \ref{cor: ChiangWolfgeneral}), and Byrne and Weger (Theorem \ref{thm: ByrneWeger}) all coincide. To compare the bound (\ref{eqn: Ald prime MDS}) in Theorem \ref{thm: main theorem bound} with these bounds, it is helpful to observe that   
	\begin{equation}
		\left\lfloor \mu_p(n-k+1)\cdot \frac{p-1}{p}\right\rfloor +1 = \left\lfloor \mu_p(n-k+1)-\frac{\mu_p}{p}(n-k+1)\right\rfloor +1. 
	\end{equation}
	Simple calculations show that:
	\begin{itemize}
		\item If $\mu_p(n-k+1)$ is an integer, then the bound (\ref{eqn: Ald prime MDS}) subsumes the other bounds and offers a strict improvement when $n\geq k+3$.
		\item If $\mu_p(n-k+1)$ is not an integer, then the bound (\ref{eqn: Ald prime MDS}) still subsumes the other bounds and provides a strict improvement when $n\geq k+5$.
	\end{itemize}
\end{rem}

It is straightforward to verify that if the bound (\ref{eqn: Ald prime not MDS}) exceeds that in (\ref{eqn: Ald prime MDS}), then $\C$ is not MDS. Thus, taking into account Corollaries \ref{cor: bound for MDS prime} and \ref{cor: floor bound prime case}, along with Theorem \ref{thm: main theorem bound}, we obtain the optimal application the new bounds in Table \ref{tab:optimal_bounds}. These observations are also reflected in the numerical comparisons presented in Figures \ref{fig:Prime_n_3_7}-\ref{fig:Prime_n_10_5} in Section \ref{sec: comparisons}.


\begin{table}[h]
	\centering
	\renewcommand{\arraystretch}{1.5}
	\begin{tabular}{|c|c|}
		\hline
		\textbf{Bound}:  $\Phi(n,k,p)\le $ & \textbf{Optimal Conditions} \\
		\hline
		$\displaystyle \left\lfloor \mu_p(n-k+1)\cdot \frac{p-1}{p}\right\rfloor +1$ & $k+1\leq n\leq p+1$ \\[8pt]
		\hline
		$\displaystyle \mu_p(n-k)$ & $n>p+1$ and $k\leq p$, or $n>k+1$ and $k\geq p$ \\[4pt]
		\hline
		$\displaystyle \mu_p\left(n-k+1- \left\lfloor \frac{n-k+1}{p+1}  \right\rfloor\right)$ & $n>2p+k+1$ \\[8pt]
		\hline
	\end{tabular}
	\caption{Optimal application of the presented bounds for $\Phi(n,k,p)$ when $k\geq 2$ and $p$ is prime.}
	\label{tab:optimal_bounds}
\end{table}	
%

\begin{example}
	Consider the $(5,2,4)_5$-MDS code $\C$ with generator matrix
	\[
	G = \begin{bmatrix}
		0	& 1 & 2 & 2 & 1  \\
		2	& 1 & 4 & 1 & 4 
	\end{bmatrix}.
	\] 
	This code satisfies $d_L(\C) = 5$. Although $\C$ does not meet the bound in Theorem \ref{thm: ByrneWeger}, Theorem \ref{thm: main theorem bound} confirms that $\C$ is in fact an MLDR code.
	
	Furthermore, the parameters of this code satisfy $\lfloor\mu_p(n-k)\rfloor = 4$, demonstrating that the conditions required in Corollary \ref{cor: bound for MDS prime} are necessary.
\end{example}

We next explore some elementary properties of $\Phi(n,k,p)$. The following lemma will be useful.

\begin{lemma} \label{lem: interval min}
	Let $k\ge 2$ be an integer, $m>0$. If $A\subset (0,m]$ with $|A|=k$, and $B= A\cup \{|a_i-a_j| \mid a_i,a_j\in A\}$, then there exists $x\in B$ with $0<x\le \frac{m}{k}$.    	
\end{lemma}

\begin{proof}
	Assume for contradiction that $B \cap (0, \frac{m}{k}] = \emptyset$. The interval $(0, m]$ is partitioned into $k$ subintervals of the form  $\left( \frac{m \cdot i}{k}, \frac{m \cdot (i+1)}{k} \right], \quad 0 \leq i \leq k-1.$ 	By the pigeonhole principle, at least one of these subintervals must contain at least two elements of $A$, contradicting our assumption that $B \cap (0, \frac{m}{k}] = \emptyset$.
\end{proof}

\begin{lemma}\label{lem: properties of Phi}
	Some elementary properties of $\Phi(n,k,p)$, $p$ prime.
	
	\begin{enumerate}
		\item \label{pt: n=k} $\Phi(k,k,p) =1$
		\item \label{pt: n+1 vs n} $\Phi(n,k,p)\le \Phi(n+1,k,p) $
		\item \label{pt: k+1 vs k} $\Phi(n,k,p)\ge \Phi(n,k+1,p) $
		\item  \label{pt: phi bush 2} If $p\in \{2,3\}$, or if $k\ge \frac{p-1}{2}$ then $\Phi(k+1,k,p) =2$. 
		\item  \label{pt: phi bush 3} If $p\ge 5$ and $k< \frac{p-1}{2}$ then $\Phi(k+1,k,p) \le 2+\frac{p-3}{2k} $.  
		\item  \label{pt: Wood 1-d codes} $\Phi(n,1,p) \le n\mu_p$ with equality only if $n=\alpha\cdot \frac{p-1}{2}$ for some integer $\alpha>0$.
	\end{enumerate}
\end{lemma}
\begin{proof}
	\ref{pt: n=k}: See Example \ref{example: n=k n=k+1, prime}. \\
	\ref{pt: n+1 vs n}: If $\C$ achieves $\Phi(n,k,p)$, then repeating any coordinate results in code $\C'$ of length $n+1$, and $d_L(\C')\ge d_L(\C)$.\\
	\ref{pt: k+1 vs k}: Suppose $\C$ achieves $\Phi(n,k+1,p)$ and let $v$ be a codeword of minimum Lee weight in $\C$. Any linearly independent $k$-subset $S'$ with $\{v\} \subseteq S'\subseteq \C $ spans a $k$-dimensional code $\C'$ with $w_L(\C')\ge w_L(v) $.\\  
	\ref{pt: phi bush 2}: For $p\in \{2,3\}$ see Example \ref{example: n=k n=k+1, prime}. Assume  $k\ge \frac{p-1}{2}$. If $\C$ is not MDS, then $d_H(\C)=1=d_L(\C)$, so assume $\C$ is MDS. The classical result of Bush \cite{Bush1952} gives that $\C$ is Lee-equivalent (up to positional permutations, and multiplication of a particular coordinate position by $\pm 1$) to a code with generator matrix $G = [I_k\mid\bar{v}]$, where  $w_H(\bar{v}) = k$. If $k=\frac{p-1}{2}$, then either some coordinate of $\bar{v}$ has Lee weight $1$ (and $d_L(\C)=2$), or at least two coordinates of $\bar{v}$ have the same Lee weight. If $k>\frac{p-1}{2}$ then it must be the case that two coordinates of $\bar{v}$ have the same Lee weight. The difference or sum of the corresponding rows give  $d_L(\C) \le 2$, and observing that $d_L(\C) \ge d_H(\C)=2$ gives the result.\\
	\ref{pt: phi bush 3}: Arguing as in part \ref{pt: phi bush 2}, we may assume $\C$ is MDS, and by multiplying rows by $-1$ as appropriate, we may assume $\C$ has a generator matrix $G' = [A\mid \bar{v}]$, where $A$ is a diagonal matrix with each diagonal entry being $\pm 1$, $w_H(\bar{v}) = k$, and the entries $v_i$ of $\bar{v}$ are mutually distinct, satisfying $2\le v_i\le \frac{p-1}{2}$ (else $d_L(\C) =2$). From the Lemma \ref{lem: interval min}, among the row of $G'$, or the pairwise differences of rows of $G'$ there exist a codeword of Lee weight at most $2+\frac{p-3}{2k}$. \\ 
	\ref{pt: Wood 1-d codes}: Let $\C$ be an $[n,1,d]_p$ code. An averaging argument shows $d_L(\C)\le n\mu_p$, with equality if and only if all nonzero codewords have constant Hamming weight $n$, and constant Lee weight $n\mu_p$. In  \cite{Wood2002}, Wood showed that a shortest length constant Lee-weight linear code over $\Z_p$ or $\Z_{2^t}$ provides all constant Lee-weight codes in that all constant Lee-weight codes are $m$-fold replications of the shortest such code (up to Lee equivalence). In particular, he showed that any constant Lee weight code over $\Z_p$ is an $m$-fold replication of $\vspan{(1,2,\ldots,\frac{p-1}{2})}$ up to multiplication of each coordinate by $\pm 1$, (and  any constant Lee weight code over $\Z_{2^t}$ is an $m$-fold replication of $\vspan{(1,2,\ldots,2^t-1)}$ up to multiplication of each coordinate by $\pm 1$). The result follows.     
\end{proof}

\begin{lemma}\label{lem: short code cannot meet mu bound}
	Let $p$ be prime. 	If $\C$ is an $[n,k,d]_p$ code with $p > 2$ and $d_L(\C) = d \cdot \mu_p$, then $\frac{p-1}{2}$ divides $d$.\\ 
	In particular, if $p \equiv 3 \mod 4$ or if $d$ is even, and $\frac{p-1}{2}$ does not divide $d$, then 
	\[
	d_L(\C) \leq d \cdot \mu_p - 1.
	\]
\end{lemma}

\begin{proof}
	Let $\C$ be an $[n,k,d]_p$ code. An averaging argument gives $w_L(\C) \leq d\mu_p$, with equality if and only if all minimum Hamming weight codewords have constant Lee weight $d\mu_p$. Using Wood's result as in part \ref{pt: Wood 1-d codes} of Lemma \ref{lem: properties of Phi} completes the proof.
\end{proof}

\begin{corollary}
	Let $p$ be prime. 	If $n \leq k + \frac{p-5}{2}$, then 
	\[
	\Phi(n,k,p) < \mu_p(n-k+1).
	\]
\end{corollary}

\begin{proof}
	Since $n \geq k$, we necessarily have $p \geq 5$. Let $\C$ be an $[n,k,d]_p$ code. If $d \leq n-k$, then the result follows from Corollary \ref{cor: bound for AsMDS prime}. If $d = n-k+1$, then the assumption provides $d \leq \frac{p-3}{2}$, so the result follows from Lemma \ref{lem: short code cannot meet mu bound}.
\end{proof}

Having explored new bounds for codes over $\Z_p$, we now turn our attention to codes over $\Z_q$, where $q = p^t$.

\subsection{Bounds for Linear Codes Over $\Z_q$, $q=p^t$}

The following lemma enables bounds established for codes over $\Z_p$ to be extended to bounds on codes over  $\Z_q$, $q=p^t$.  

\begin{lemma}\label{lem: Phi p^t} 
	Let $p$ be prime. 	If $t \geq 1$, then
	\begin{equation} \label{eqn: Phi p^t}
		\Phi(n,K,p^t) \leq p^{t-1} \cdot \Phi(n,K,p).
	\end{equation}
\end{lemma}
\begin{proof}
	Let $\mathcal{C} \subseteq\left(\Z_{p^t}\right)^n$ be an $[n,k,d]_q$ code of rank $K$, and let $S(\C)$ denote the socle of $\C$, so $S(\C)$ is an $[n,K,d]_p$-code over the alphabet $ \vspan{p^{t-1}}$. It follows that \[d_L(\C)\le d_L(S(\C))\le p^{t-1}\Phi(n,K,p).\]
\end{proof}

The following examples show that equality in (\ref{eqn: Phi p^t}) is attainable.

\begin{example}\label{example: n=k, prime power}
	For fixed $K \geq 1$, consider the code $\C$ over $\Z_{p^t}$ with generator matrix $p^{t-1} I_K$. 
	Since, for any $1 \leq \alpha \leq p-1$, we have $w_L(\alpha \cdot p^{t-1}) \leq p^{t-1}$, it follows immediately that $d_L(\C) = p^{t-1}$. 
	With reference to part \ref{pt: n=k} of Lemma \ref{lem: properties of Phi}, we obtain equality in (\ref{eqn: Phi p^t}), so
	\[
	\Phi(K,K,p^t) = p^{t-1}.
	\]
\end{example}
\begin{example}\label{example: n=k+1, prime power}
	With $G_1$ as in Example \ref{example: n=k n=k+1, prime}, the code $\C_1$ over $\Z_{p^t}$ with generator matrix $p^{t-1} G_1$ satisfies $d_L(\C_1) = 2 \cdot p^{t-1}$, giving
	
	\[\Phi(K+1,K,p^t) \ge 2 \cdot p^{t-1}.\]
	
	Consequently, if $p \in \{2,3\}$, we have equality in (\ref{eqn: Phi p^t}), and 
	\[
	\Phi(K+1,K,p^t) = 2 \cdot p^{t-1}.
	\]
\end{example}

We are now able to state our main result.

\begin{theorem} \label{thm: main thm}
	Let $p$ be prime. If  $t \geq 1$, and $K \geq 2$, then the following bounds hold:
	\begin{itemize}
		\item[] 
		If $n \geq K$, then 
		\begin{equation} \label{eqn: Main general}   \tag{A}
			\Phi(n,K,p^t) 	\leq  p^{t-1} \left\lfloor \mu_p\left(n-K+1- \left\lfloor \frac{n-K+1}{p+1}  \right\rfloor\right)\right\rfloor.
		\end{equation}
		
		\item[]  
		If $K \geq p$ and $n > K+1$, or if $2 \leq K \leq p$ and $n > p+1$, then  
		\begin{equation} \label{eqn: Main NotMDS} \tag{B}
			\Phi(n,K,p^t) \leq p^{t-1} \left\lfloor \mu_p(n-K) \right\rfloor.
		\end{equation}   
		
		\item[]  
		If $3 \leq K+1 \leq n \leq p+1$, then  
		\begin{equation} \label{eqn: Main MDS} \tag{C}
			\Phi(n,K,p^t) \leq p^{t-1} \cdot \left\lfloor \mu_p(n-K+1) \cdot \frac{p-1}{p} + 1 \right\rfloor.
		\end{equation}
		
		\item[] In the case $n = K$, we have 
		\begin{equation}  \label{eqn: Main MDS n less than K} \tag{D}
			\Phi(K,K,p^t) = p^{t-1}.
		\end{equation}	
	\end{itemize}
\end{theorem}  

\begin{proof}
	Appealing to Lemma \ref{lem: Phi p^t},   (\ref{eqn: Main general})   follows from Corollary \ref{cor: floor bound prime case},   (\ref{eqn: Main NotMDS}) follows from Corollary \ref{cor: bound for MDS prime}, and (\ref{eqn: Main MDS}) follow from Theorem \ref{thm: main theorem bound}. For (\ref{eqn: Main MDS n less than K}) see Example \ref{example: n=k, prime power}.
\end{proof}

In comparing the bounds in Theorem \ref{thm: main thm}, we observe that when the conditions for bound (\ref{eqn: Main NotMDS}) are satisfied, it provides a stronger bound than (\ref{eqn: Main general}) whenever $n-K+1 < p+1$. Additionally, simple calculations show that when the conditions for bound (\ref{eqn: Main MDS}) hold, it subsumes (\ref{eqn: Main general}) and offers a strict improvement in many cases, particularly when $n > K+5$. These observations are illustrated in Section \ref{sec: comparisons}, through Figures \ref{fig:Prime_n_3_7}-\ref{fig:Plot_n_11,13^2}, and Table \ref{table: bounds compared}.

Just as in the prime case, codes that are optimal with respect to the Lee metric must also be optimal with respect to the Hamming metric. The following result observed by Byrne and Weger \cite{Byrne2023} may be employed to establish this connection:

\begin{lemma}\label{lem: MDR iff MDS Socle}
	Let $q=p^t$, $p$ prime.  A linear code over $\Z_q$ is MDR if and only if its socle is classically equivalent to an $[n,K]_p$-MDS code. 
\end{lemma}

An immediate consequence of this is the following

\begin{lemma}\label{lem: LDimpliesMDS prime power}
	Let $p$ be prime. 	If $\C$ is a linear $[n,K,d]_q$ code, $q=p^t$, and $d_L(\C)>p^{t-1}\mu_p(n-K)$ then $\C$ is MDR. In particular, if $\C$ is also a free code, then $\C$ is MDS. 
\end{lemma}
\begin{proof}
	Suppose $\C$ is not MDR. Let $S(\C)$ denote the socle of $\C$.  By the Lemma \ref{lem: MDR iff MDS Socle},  $S(\C)$ is not MDS. Let $x\in S(\C)$ be a codeword of minimum Hamming weight, so in particular  $w_H(x)\le n-K$. Assume the nonzero entries are the first $w_H(x)$ entries.   Since $x\in S(\C) = \C\cap \vspan{p^{t-1}}$, $\vspan{x}$ holds $p-1$ nonzero codewords all of which are nonzero in the first $d_H$ coordinates, and no two agreeing in any of the first $w_H(x)$ coordinates. Averaging as in Lemma \ref{lem:LDimpliesMDS}, and recognizing that each coordinate in $S(\C)$ is an element of $\vspan{p^{t-1}}$, we obtain 
	\[
	d_L(\C)\le p^{t-1}\mu_p(n-K).
	\]       
	The result follows.
\end{proof}

\subsection{Linear Codes Over $\Z_q$, $q=2^t$}

It is well known (see e.g. \cite{MR0465509}) that the only binary linear MDS codes are the trivial ones, namely the $[n,1,n]_2 $  repetition  codes, the  $ [n,n,1]_2 $  universe  codes, and the $  [k+1,k,2]_2 $ single-parity-check codes. This allows us to show the following.

\begin{lemma}\label{lem: MDR even case}
	Let $\C$ be a linear $[n,K,d]_q$ MDR code, $q=2^t$. One of the following must hold:
	\begin{itemize}
		\item $K=1$ and  $d_L(\C)\le n\cdot 2^{t-1}$
		\item $n=K>1$ and $d_L(\C) \le 2^{t-1}$
		\item  $n=K+1$ and $d_L(\C) \le 2^{t}$
	\end{itemize}
\end{lemma}
\begin{proof}
	Since $\C$ is MDR, the socle, $S(\C)$ is a binary $[n,k,d]_2$-MDS code over the alphabet $\{2^{t-1},0\}$.  All such codes are classified, giving the result.
\end{proof}

\begin{corollary}\label{lem: LDimpliesMDS even prime power}
	If  $n>K+1$,  then $\Phi(n,K,2^t)\le 2^{t-1}(n-K)$ \end{corollary}

\begin{proof} 	
	Follows from Lemma \ref{lem: MDR even case}.  
\end{proof}

\section{Some Contextual Comparisons}\label{sec: comparisons}

We now present numerical comparisons of the bounds discussed in this work. In the following plots, we exclude bounds that yield significantly higher upper estimates for $\Phi(n,K,p^t)$, as our focus is on obtaining the tightest upper bounds.

Some of the existing bounds on $\Phi(n,k,p^t)$ presented in Section \ref{subsec: Established bounds} are most effective when $t=1$, in which case all codes are free codes. Given this, we begin by comparing bounds over $\mathbb{Z}_p$.

\subsection*{Numerical Comparisons Over $\Z_p$}

When working over $\Z_p$, the bounds in Corollaries \ref{cor: Chiang and Wolf} and \ref{cor: ChiangWolfgeneral}, and Theorem \ref{thm: ByrneWeger} coincide. Consequently, in the comparisons that follow, we use the bound in Corollary \ref{cor: Chiang and Wolf} (due to Chiang and Wolf) as a representative for all three.\\   

\begin{minipage}{.5\textwidth}
	\begin{tikzpicture}
		\begin{axis}[
			width=8.5cm, height=8cm,
			xlabel={$n$},
			ylabel={ $\Phi(n,3,7)\le $},
			title={},
			legend pos=north west,
			y post scale=1,
			x post scale=1,
			grid=major
			]
			
			\addplot[smooth,mark=otimes,orange] coordinates {(15,24) (16,26) (17,28) (18,30) (19,32) (20,34) (21,36) (22,38) (23,40) (24,42) (25,44) (26,46) (27,48) (28,50) (29,52) (30,54) (31,56) (32,58) (33,60) (34,62) };
			\addlegendentry{Thm.  \ref{thm: main thm} (B) }
			\addplot[smooth,mark=triangle,green] coordinates {(15,26) (16,28) (17,30) (18,32) (19,34) (20,36) (21,38) (22,40) (23,42) (24,44) (25,46) (26,48) (27,50) (28,52) (29,54) (30,56) (31,58) (32,60) (33,62) (34,64) };
			\addlegendentry{Cor. \ref{cor: Chiang and Wolf} }
			\addplot[smooth,mark=star,black] coordinates {(15,25) (16,27) (17,29) (18,30) (19,32) (20,34) (21,36) (22,37) (23,39) (24,41) (25,42) (26,44) (27,46) (28,48) (29,49) (30,51) (31,53) (32,55) (33,56) (34,58) };
			\addlegendentry{Cor. \ref{cor: Wyner and Graham MLDR}}
			\addplot[smooth,mark=*,orange] coordinates {(15,24) (16,26) (17,28) (18,28) (19,30) (20,32) (21,34) (22,36) (23,38) (24,40) (25,42) (26,42) (27,44) (28,46) (29,48) (30,50) (31,52) (32,54) (33,56) (34,56) };
			\addlegendentry{Thm. \ref{thm: main thm} (A)}	
		\end{axis}
	\end{tikzpicture}
	\captionof{figure}{Bounds on $\Phi(n,3,7)$} \label{fig:Prime_n_3_7}  \vspace*{5mm}
	\vspace*{1cm}
\end{minipage}
\begin{minipage}{.5\textwidth}
	\begin{tikzpicture}
		\begin{axis}[
			width=8.5cm, height=8cm,
			xlabel={$n$},
			ylabel={$\Phi(n,4,37)\le$},
			title={},
			legend pos=north west,
			y post scale=1,
			grid=major
			]
			
			\addplot[smooth,mark=o,orange] coordinates {(24,195) (25,204) (26,213) (27,222) (28,232) (29,241) (30,250) (31,259) (32,269) (33,278) (34,287) (35,296) (36,306) (37,315) };
			\addlegendentry{Thm. \ref{thm: main thm} (C)}
			\addplot[smooth,mark=triangle,green] coordinates {(24,199) (25,209) (26,218) (27,228) (28,237) (29,247) (30,256) (31,266) (32,275) (33,285) (34,294) (35,304) (36,313) (37,323) };
			\addlegendentry{Cor. \ref{cor: Chiang and Wolf}, Thm. \ref{thm: main thm} (A)}
			\addplot[smooth,mark=star,black] coordinates {(24,221) (25,231) (26,240) (27,249) (28,258) (29,268) (30,277) (31,286) (32,295) (33,305) (34,314) (35,323) (36,332) (37,342) };
			\addlegendentry{Cor. \ref{cor: Wyner and Graham MLDR}}
		\end{axis}
	\end{tikzpicture}
	\captionof{figure}{Bounds on $\Phi(n,4,37)$} \label{fig:Prime_n_4_37} \vspace*{5mm}
	\vspace*{1cm}
\end{minipage}

\begin{minipage}{.5\textwidth}
	\begin{tikzpicture}
		\begin{axis}[
			width=8.5cm, height=8cm,
			xlabel={$K$},
			ylabel={$\Phi(2K,K,5)\le$},
			title={},
			legend pos=north west,
			y post scale=1,
			grid=major
			]
			
			\addplot[smooth,mark=otimes,orange] coordinates {(3,12) (5,15) (7,18) (9,21) (11,24) (13,27) (15,30) (17,33) (19,36) (21,39) (23,42) (25,45) (27,48) (29,51) (31,54) (33,57) (35,60) (37,63) (39,66) };
			\addlegendentry{Thm. \ref{thm: main thm} (B)}
			\addplot[smooth,mark=triangle,green] coordinates {(3,13) (5,16) (7,19) (9,22) (11,25) (13,28) (15,31) (17,34) (19,37) (21,40) (23,43) (25,46) (27,49) (29,52) (31,55) (33,58) (35,61) (37,64) (39,67) };
			\addlegendentry{Cor. \ref{cor: Chiang and Wolf}}
			\addplot[smooth,mark=x,blue] coordinates {(3,16) (5,20) (7,24) (9,28) (11,32) (13,36) (15,40) (17,44) (19,48) (21,52) (23,56) (25,60) (27,64) (29,68) (31,72) (33,76) (35,80) (37,84) (39,88) };
			\addlegendentry{Cor. 6}
			\addplot[smooth,mark=star,black] coordinates {(3,13) (5,18) (7,22) (9,27) (11,32) (13,37) (15,42) (17,46) (19,51) (21,56) (23,61) (25,66) (27,70) (29,75) (31,80) (33,85) (35,90) (37,94) (39,99) };
			\addlegendentry{Cor. \ref{cor: Wyner and Graham MLDR}}
			\addplot[smooth,mark=*,orange] coordinates {(3,12) (5,15) (7,16) (9,19) (11,22) (13,24) (15,27) (17,30) (19,31) (21,34) (23,37) (25,39) (27,42) (29,45) (31,46) (33,49) (35,52) (37,54) (39,57) };
			\addlegendentry{Thm. \ref{thm: main thm} (A)}
			
		\end{axis}
	\end{tikzpicture}
	\captionof{figure}{Bounds on $\Phi(2K+5,K,5)$} \label{fig:Prime_2K_K_5} \vspace*{5mm}
\end{minipage}
\begin{minipage}{.5\textwidth}
	\begin{tikzpicture}
		\begin{axis}[
			width=8.5cm, height=8cm,
			xlabel={$n$},
			ylabel={$\Phi(n,10,5)\le $},
			title={},
			legend pos=north west,
			y post scale=1,
			grid=major
			]
			
			\addplot[smooth,mark=otimes,orange] coordinates {(25,22) (26,24) (27,25) (28,27) (29,28) (30,30) (31,31) (32,33) (33,34) (34,36) (35,37) (36,39) (37,40) (38,42) (39,43) };
			\addlegendentry{Thm. \ref{thm: main thm} (B)}
			\addplot[smooth,mark=triangle,green] coordinates {(25,24) (26,25) (27,27) (28,28) (29,30) (30,31) (31,33) (32,34) (33,36) (34,37) (35,39) (36,40) (37,42) (38,43) (39,45) };
			\addlegendentry{Cor. \ref{cor: Chiang and Wolf}}
			\addplot[smooth,mark=star,black] coordinates {(25,30) (26,31) (27,32) (28,33) (29,34) (30,36) (31,37) (32,38) (33,39) (34,40) (35,42) (36,43) (37,44) (38,45) (39,46) };
			\addlegendentry{Cor. \ref{cor: Wyner and Graham MLDR}}
			\addplot[smooth,mark=*,orange] coordinates {(25,21) (26,22) (27,22) (28,24) (29,25) (30,27) (31,28) (32,30) (33,30) (34,31) (35,33) (36,34) (37,36) (38,37) (39,37) };
			\addlegendentry{Thm. \ref{thm: main thm} (A)}
			
		\end{axis}
	\end{tikzpicture}
	\captionof{figure}{Bounds on $\Phi(n,10,5)$} \label{fig:Prime_n_10_5}  \vspace*{5mm}
\end{minipage}

Figures \ref{fig:Prime_n_3_7} to \ref{fig:Prime_n_10_5}  illustrate that the bounds obtained from Theorem \ref{thm: main thm} consistently provide the least of all values when compared to the existing bounds, representing improved estimates of $\Phi(n, k, p)$. Next, we compare bounds on MLDR codes over $\Z_p^t$, $t>1$.

\subsection*{Numerical Comparisons of Bounds on $\Phi(n,K,p^t)$, $t>1$}

The plots in Figures \ref{fig:Plot_2K_K_3^5}-\ref{fig:Plot_n_20_5^3}, as well at the values in Table \ref{table: bounds compared} illustrate that in the case $q=p^t$ with $t>1$, the bounds obtained from Theorem \ref{thm: main thm} consistently yield the least values.

\begin{minipage}{.5\textwidth}
	\begin{tikzpicture}
		\begin{axis}[
			width=8.5cm, height=8cm,
			xlabel={$K$},
			ylabel={ $\Phi(2K,K,3^5)\le $},
			title={},
			legend pos=north west,
			y post scale=1,
			grid=major
			]
			
			\addplot[smooth,mark=otimes,orange] coordinates {(5,405) (6,486) (7,567) (8,648) (9,729) (10,810) (11,891) (12,972) (13,1053) (14,1134) (15,1215) (16,1296) (17,1377) (18,1458) (19,1539) };
			\addlegendentry{Thm. \ref{thm: main thm} (B)}
			\addplot[smooth,mark=triangle,red] coordinates {(5,486) (6,567) (7,648) (8,729) (9,810) (10,891) (11,972) (12,1053) (13,1134) (14,1215) (15,1296) (16,1377) (17,1458) (18,1539) (19,1620) };
			\addlegendentry{Thm. \ref{thm: ByrneWeger}}
			\addplot[smooth,mark=triangle,green] coordinates {(5,610) (6,732) (7,854) (8,976) (9,1098) (10,1159) (11,1281) (12,1403) (13,1525) (14,1647) (15,1708) (16,1830) (17,1952) (18,2074) (19,2196) };
			\addlegendentry{Cor. \ref{cor: ChiangWolfgeneral}}
			\addplot[smooth,mark=star,black] coordinates {(5,610) (6,729) (7,850) (8,972) (9,1093) (10,1215) (11,1336) (12,1457) (13,1579) (14,1700) (15,1822) (16,1943) (17,2065) (18,2186) (19,2308) };
			\addlegendentry{Cor. \ref{cor: Wyner and Graham MLDR}}
			\addplot[smooth,mark=*,orange] coordinates {(5,405) (6,486) (7,486) (8,567) (9,648) (10,729) (11,729) (12,810) (13,891) (14,972) (15,972) (16,1053) (17,1134) (18,1215) (19,1215) };
			\addlegendentry{Thm. \ref{thm: main thm} (A)}

		\end{axis}
	\end{tikzpicture}
	\captionof{figure}{Bounds on $\Phi(2K,K,3^5)$} \label{fig:Plot_2K_K_3^5}
\end{minipage}
\begin{minipage}{.5\textwidth}
	\begin{tikzpicture}
		\begin{axis}[
			width=8.5cm, height=8cm,
			xlabel={$K$},
			ylabel={$\Phi(2K,K,2^4)\le $},
			title={},
			legend pos=north west,
			y post scale=1,
			grid=major
			]
			
			\addplot[smooth,mark=triangle,red] coordinates {(4,40) (5,48) (6,56) (7,64) (8,72) (9,80) (10,88) (11,96) (12,104) (13,112) (14,120) (15,128) (16,136) (17,144) (18,152) (19,160) };
			\addlegendentry{Thm. \ref{thm: ByrneWeger}}
			\addplot[smooth,mark=triangle,green] coordinates {(4,34) (5,42) (6,51) (7,59) (8,64) (9,72) (10,81) (11,89) (12,93) (13,102) (14,110) (15,119) (16,123) (17,132) (18,140) (19,149) };
			\addlegendentry{Cor. \ref{cor: ChiangWolfgeneral}}
			\addplot[smooth,mark=star,black] coordinates {(4,34) (5,41) (6,48) (7,56) (8,64) (9,72) (10,80) (11,88) (12,96) (13,104) (14,112) (15,120) (16,128) (17,136) (18,144) (19,152) };
			\addlegendentry{Cor. \ref{cor: Wyner and Graham MLDR}}
			\addplot[smooth,mark=*,orange] coordinates {(4,32) (5,32) (6,40) (7,48) (8,48) (9,56) (10,64) (11,64) (12,72) (13,80) (14,80) (15,88) (16,96) (17,96) (18,104) (19,112) };
			\addlegendentry{Thm. \ref{thm: main thm} (A)}
			
		\end{axis}
	\end{tikzpicture}
	\captionof{figure}{Bounds on $\Phi(2K,K,2^4)$} \label{fig:Plot_2K_K_2^4}
\end{minipage}

\begin{minipage}{.5\textwidth}
	\begin{tikzpicture}
		\begin{axis}[
			width=8.5cm, height=8cm,
			xlabel={$K$},
			ylabel={$\Phi(\left\lfloor \frac{3K}{2}\right\rfloor,K,2^2)\le$},
			title={},
			legend pos=north west,
			y post scale=1,
			grid=major
			]
			
			\addplot[smooth,mark=otimes,orange] coordinates {(4,4) (5,4) (6,6) (7,6) (8,8) (9,8) (10,10) (11,10) (12,12) (13,12) (14,14) (15,14) (16,16) (17,16) (18,18) (19,18) };
			\addlegendentry{Thm. \ref{thm: main thm} (B)}
			\addplot[smooth,mark=triangle,red] coordinates {(4,6) (5,6) (6,8) (7,8) (8,10) (9,10) (10,12) (11,12) (12,14) (13,14) (14,16) (15,16) (16,18) (17,18) (18,20) (19,20) };
			\addlegendentry{Thm. \ref{thm: ByrneWeger}}
			\addplot[smooth,mark=triangle,green] coordinates {(4,6) (5,8) (6,9) (7,10) (8,12) (9,13) (10,14) (11,16) (12,17) (13,18) (14,20) (15,21) (16,22) (17,24) (18,25) (19,26) };
			\addlegendentry{Cor. \ref{cor: ChiangWolfgeneral}}
			\addplot[smooth,mark=star,black] coordinates {(4,6) (5,7) (6,9) (7,10) (8,12) (9,13) (10,15) (11,16) (12,18) (13,19) (14,21) (15,22) (16,24) (17,25) (18,27) (19,28) };
			\addlegendentry{Cor. \ref{cor: Wyner and Graham MLDR}}
			\addplot[smooth,mark=*,orange] coordinates {(4,4) (5,4) (6,6) (7,6) (8,8) (9,8) (10,8) (11,8) (12,10) (13,10) (14,12) (15,12) (16,12) (17,12) (18,14) (19,14) };
			\addlegendentry{Thm. \ref{thm: main thm} (A)}
			
		\end{axis}
	\end{tikzpicture}
	\captionof{figure}{Bounds on $\Phi(\left\lfloor \frac{3K}{2}\right\rfloor,K,2^2)$} \label{fig:Plot_3K/2_K_2^2}\vspace*{5mm}
\end{minipage}
\begin{minipage}{.5\textwidth}
	\begin{tikzpicture}
		\begin{axis}[
			width=8cm, height=8cm,
			xlabel={$n$},
			ylabel={$\Phi(n,20,5^3)\le $},
			title={},
			legend pos=north west,
			y post scale=1,
			grid=major
			]
			
			\addplot[smooth,mark=otimes,orange] coordinates {(110,3375) (111,3400) (112,3450) (113,3475) (114,3525) (115,3550) (116,3600) (117,3625) (118,3675) (119,3700) (120,3750) (121,3775) (122,3825) (123,3850) (124,3900) (125,3925) (126,3975) (127,4000) (128,4050) (129,4075) };
			\addlegendentry{Thm. \ref{thm: main thm} (B)}
			\addplot[smooth,mark=triangle,red] coordinates {(110,3412) (111,3450) (112,3487) (113,3525) (114,3562) (115,3600) (116,3637) (117,3675) (118,3712) (119,3750) (120,3787) (121,3825) (122,3862) (123,3900) (124,3937) (125,3975) (126,4012) (127,4050) (128,4087) (129,4125) };
			\addlegendentry{Thm. \ref{thm: ByrneWeger}}
			\addplot[smooth,mark=triangle,green] coordinates {(110,3307) (111,3339) (112,3370) (113,3402) (114,3433) (115,3465) (116,3496) (117,3528) (118,3559) (119,3591) (120,3622) (121,3654) (122,3685) (123,3717) (124,3748) (125,3780) (126,3811) (127,3843) (128,3874) (129,3906) };
			\addlegendentry{Cor. \ref{cor: ChiangWolfgeneral}}
			\addplot[smooth,mark=x,black] coordinates {(110,3437) (111,3468) (112,3499) (113,3531) (114,3562) (115,3593) (116,3624) (117,3656) (118,3687) (119,3718) (120,3749) (121,3781) (122,3812) (123,3843) (124,3874) (125,3906) (126,3937) (127,3968) (128,3999) (129,4030) };
			\addlegendentry{Cor. \ref{cor: Wyner and Graham MLDR}}
			\addplot[smooth,mark=*,orange] coordinates {(110,2850) (111,2875) (112,2925) (113,2950) (114,3000) (115,3000) (116,3025) (117,3075) (118,3100) (119,3150) (120,3175) (121,3175) (122,3225) (123,3250) (124,3300) (125,3325) (126,3375) (127,3375) (128,3400) (129,3450) };
			\addlegendentry{Thm. \ref{thm: main thm} (A)}
			
		\end{axis}
	\end{tikzpicture}
	\captionof{figure}{Bounds on $\Phi(n,20,125)$} \label{fig:Plot_n_20_5^3}\vspace*{5mm}
\end{minipage}

\begin{minipage}{.5\textwidth}
	\begin{tikzpicture}
		\begin{axis}[
			width=8.5cm, height=8.5cm,
			xlabel={$n$},
			ylabel={$\Phi(n,3,13^2)\le$},
			title={},
			legend pos=north west,
			y post scale=1,
			grid=major
			]
			
			\addplot[smooth,mark=o,orange] coordinates {(9,299) (10,338) (11,390) (12,429) (13,468) (14,507) };
			\addlegendentry{Thm. \ref{thm: main thm} (C)}
			\addplot[smooth,mark=triangle,red] coordinates {(9,318) (10,364) (11,409) (12,455) (13,500) (14,546) };
			\addlegendentry{Thm. \ref{thm: ByrneWeger}}
			\addplot[smooth,mark=triangle,green] coordinates {(9,382) (10,425) (11,467) (12,510) (13,552) (14,595) };
			\addlegendentry{Cor. \ref{cor: ChiangWolfgeneral}}
			\addplot[smooth,mark=x,black] coordinates {(9,380) (10,422) (11,464) (12,507) (13,549) (14,591) };
			\addlegendentry{Cor. \ref{cor: Wyner and Graham MLDR}}
			\addplot[smooth,mark=*,orange] coordinates {(9,312) (10,364) (11,403) (12,455) (13,494) (14,546) };
			\addlegendentry{Thm. \ref{thm: main thm} (A)}
			
		\end{axis}
	\end{tikzpicture}
	\captionof{figure}{Bounds on $\Phi(n,3,13^2)$} \label{fig:Plot_n_3_13^2}
\end{minipage}
\begin{minipage}{.5\textwidth}
	\begin{tikzpicture}
		\begin{axis}[
			width=8.5cm, height=8.5cm,
			xlabel={$n$},
			ylabel={$\Phi(n,11,13^2)\le$},
			title={},
			legend pos=north west,
			y post scale=1,
			grid=major
			]
			
			\addplot[smooth,mark=otimes,orange] coordinates {(15,182) (16,221) (17,273) (18,312) (19,364) (20,403) (21,455) (22,494) (23,546) (24,585) };
			\addlegendentry{Thm. \ref{thm: main thm} (B)}
			\addplot[smooth,mark=triangle,red] coordinates {(15,227) (16,273) (17,318) (18,364) (19,409) (20,455) (21,500) (22,546) (23,591) (24,637) };
			\addlegendentry{Thm. \ref{thm: ByrneWeger}}
			\addplot[smooth,mark=triangle,green] coordinates {(15,467) (16,510) (17,552) (18,595) (19,637) (20,680) (21,722) (22,765) (23,807) (24,850) };
			\addlegendentry{Cor.  \ref{cor: ChiangWolfgeneral}}
			\addplot[smooth,mark=x,black] coordinates {(15,633) (16,675) (17,718) (18,760) (19,802) (20,844) (21,887) (22,929) (23,971) (24,1013) };
			\addlegendentry{Cor. \ref{cor: Wyner and Graham MLDR}}
			\addplot[smooth,mark=*,orange] coordinates {(15,221) (16,273) (17,312) (18,364) (19,403) (20,455) (21,494) (22,546) (23,585) (24,585) };
			\addlegendentry{Thm. \ref{thm: main thm} (A)}
			
		\end{axis}
	\end{tikzpicture}
	\captionof{figure}{Bounds on $\Phi(n,11,13^2)$} \label{fig:Plot_n_11,13^2}
\end{minipage}

\begin{table}[h]
	\centering	
	\caption{Comparison of upper bounds on $\Phi(n,K,p^t)$}\label{table: bounds compared}
	\label{tab:function_values}
	\begin{tabular}{lccccccc}
		\hline
		 $(n, K, p^t)$ & Cor. \ref{cor: Wyner and Graham MLDR} & Cor. \ref{cor: ChiangWolfgeneral}& Cor. \ref{cor: maintheoremgeneralAH} & Thm. \ref{thm: ByrneWeger} & Thm.  \ref{thm: main thm} (A) & Thm.  \ref{thm: main thm} (B) & Thm.  \ref{thm: main thm} (C) \\
		\hline
		(4, 2, 4) & 5 & 5 & 6 & 6 & 4 & 4 & - \\
		(4, 3, 4) & 4 & 5 & 6 & 4 & 4 & - & - \\
		(5, 3, 4) & 5 & 6 & 8 & 6 & 4 & 4 & - \\
		(6, 3, 4) & 6 & 8 & 10 & 8 & 6 & 6 & - \\
		(5, 3, 9) & 11 & 12 & 16 & 9 & 9 & 6 & - \\
		(6, 3, 9) & 13 & 15 & 20 & 12 & 9 & 9 & - \\
		(7, 3, 9) & 16 & 17 & 24 & 15 & 12 & 12 & - \\
		(28, 3, 9) & 64 & 70 & 108 & 78 & 60 & 75 & - \\
		(12, 3, 11) & 32 & 30 & 45 & 30 & 30 & - & 28 \\
		(13, 3, 11) & 35 & 33 & 50 & 33 & 33 & 30 & - \\
		(15, 3, 11) & 40 & 39 & 60 & 39 & 36 & 36 & - \\
		(5, 3, 27) & 35 & 35 & 52 & 27 & 27 & 18 & - \\
		(6, 3, 27) & 42 & 42 & 65 & 36 & 27 & 27 & - \\
		(28, 3, 27) & 196 & 196 & 351 & 234 & 180 & 225 & - \\
		(6, 3, 25) & 37 & 39 & 60 & 30 & 30 & - & 25 \\
		(6, 3, 125) & 189 & 189 & 310 & 150 & 150 & - & 125 \\
		(8, 3, 49) & 98 & 100 & 168 & 84 & 84 & - & 77 \\
		(12, 3, 121) & 363 & 366 & 660 & 330 & 330 & - & 308 \\
		\hline
	\end{tabular}
\end{table}

\section{Conclusion and Questions}

For $p$ a prime, we introduced the parameter $\Phi(n,K,p^t)$, which denotes the maximum achievable minimum Lee distance over all linear $p^t$-nary codes of length $n$ and rank $K$. Codes whose minimum Lee distance equal to $\Phi(n,K,p^t)$ are (by definition) MLDR codes. 

We established several new upper bounds on $\Phi(n,k,p^t)$. Based on analytical and numerical comparisons, the new bounds consistently yield lower values compared to the existing bounds, representing improved estimates of $\Phi(n, k, p)$. Additionally, we explored some elementary properties of this parameter in the case $t=1$ (Lemma \ref{lem: properties of Phi}).

While our results provide stronger upper bounds, an important open question remains: For $p>3$ are there infinite families of MLDR codes that achieve these new bounds? Identifying such families or proving the tightness of these bounds in general remains an interesting challenge. Additionally, exploring conditions under which the inequalities in parts \ref{pt: n+1 vs n} and \ref{pt: k+1 vs k} of Lemma \ref{lem: properties of Phi} are strict may yield further insights into the structure of MLDR codes.

\section*{Acknowledgments}
We acknowledge the support of the Natural Sciences and Engineering Research Council of Canada (NSERC), [funding reference number 2019-04103]\\
Cette recherche a \'{e}t\'{e} financ\'{e}e par le Conseil de recherches en sciences naturelles et en g\'{e}nie du Canada (CRSNG), [numéro de r\'{e}f\'{e}rence 2019-04103].

%
%
\bibliographystyle{IEEEtran}


\providecommand{\bysame}{\leavevmode\hbox to3em{\hrulefill}\thinspace}
\providecommand{\MR}{\relax\ifhmode\unskip\space\fi MR }
\providecommand{\MRhref}[2]{%
	\href{http://www.ams.org/mathscinet-getitem?mr=#1}{#2}
}
\providecommand{\href}[2]{#2}

\end{document}